\documentclass[11pt,a4paper]{article}

\usepackage[T1]{fontenc}
\usepackage[utf8]{inputenc}
\usepackage{lmodern}
\usepackage[margin=26mm]{geometry}
\usepackage{amsmath,amssymb,amsthm}
\usepackage{comment}
\usepackage{color}
\usepackage[hidelinks]{hyperref}
\usepackage{authblk}

\newcommand{\R}{\mathbb R}
\newcommand{\K}{\mathcal K}

\newcommand{\ip}[2]{\langle #1,#2\rangle}
\newtheorem{theorem}{Theorem}[section]
\newtheorem{lemma}[theorem]{Lemma}
\allowdisplaybreaks[1]

\title{On the cyclic limiting behavior of the restarted 
	Arnoldi algorithm for real symmetric matrices with restart length $s=2$}

\author{
  Vance Faber\thanks{Hoquiam, WA, USA}
  \quad
  Jarek Liesen\thanks{University of Oxford, Department of Engineering Science, \tt{jarek.liesen@reuben.ox.ac.uk}}
  \quad
  J\"org Liesen\thanks{TU Berlin, Institute of Mathematics, {\tt liesen@math.tu-berlin.de}}
  \quad
  Petr Tich\'y\thanks{Charles University, Faculty of Mathematics and Physics, {\tt petr.tichy@mff.cuni.cz}}
}

\date{14 September 2026}

\begin{document}
\maketitle
	
\begin{abstract}
We prove that the restarted Arnoldi algorithm for real symmetric matrices with restart length $s=2$ has a cyclic limiting behavior. In the terminology established in~\cite{FLT}, this is a proof of the ACI($s$) conjecture for $A=A^T\in\R^{n\times n}$ and $s=2$. As explained in~\cite{FLT}, this also proves the Forsythe conjecture~\cite{For68} about the limiting behavior of the conjugate gradient method and positive definite $A=A^T\in\R^{n\times n}$ for $s=2$. The proof we give is an extensively rewritten and simplified version of a proof found using Codex with GPT-5.6 Sol. At the end of this paper, we discuss its origin and evolution, as well as related recent developments around the Forsythe conjecture.
Our motivation is to present a proof that is instructive and transparent, so that it can enhance the knowledge of interested human readers.
\end{abstract}

\section{Introduction}

Let $A=A^T\in\R^{n\times n}$ and an integer $s\geq 1$ be given. Suppose that $v_0\in\R^n$ with $\|v_0\|=1$, where $\|\cdot\|$ is the Euclidean norm, satisfies $d(A,v_0)\geq s+1$, where $d(A,v_0)$ denotes the grade of $v_0$ with respect to $A$. The following algorithm is considered in~\cite[Section~4.1]{FLT}:
\begin{align}
	& \mbox{For $k=0,1,2,\dots$}\nonumber \\
	&\;\;\; \widetilde{v}_{k+1}=P_s(A;v_k)v_k, \label{eqn:a1}\\
	&\;\;\; v_{k+1}=\widetilde{v}_{k+1}/\|\widetilde{v}_{k+1}\|. \label{eqn:a2}
\end{align}
In \eqref{eqn:a1}, $P_s(z;v_k)$ is a monic polynomial of degree $s$ that is defined by
\begin{equation}\label{eqn:arn-proj}
	P_s(A;v_k)v_k\perp\K_s(A,v_k),  
\end{equation}
where $\K_s(A,v_k)=\operatorname{span}\{v_k,Av_k,\dots,A^{s-1}v_k\}$ is the $s$th Krylov subspace of $A$ and $v_k$. Thus, the vector $\widetilde{v}_{k+1}$ in \eqref{eqn:a1} is defined by the \emph{Arnoldi projection} of $v_k$ with respect to $A$ and $s$. The algorithm \eqref{eqn:a1}--\eqref{eqn:a2} therefore is the \emph{restarted Arnoldi algorithm} with restart length $s$ and normalization of the iterates.

It is shown in~\cite[Theorem~1]{FLT} that $d(A,v_0)\geq s+1$ implies $d(A,v_k)\geq s+1$ for all $k\geq 0$ in \eqref{eqn:a1}--\eqref{eqn:a2}. Hence the Arnoldi projection in every step $k=0,1,2,\dots$ of the algorithm is well defined, which yields a sequence of uniquely determined monic polynomials $P_s(z;v_k)$. 

Moreover, it is shown in~\cite[Theorem~6]{FLT} that if $s=1$ and $d(A,v_0)\geq 2$, then the even-index sequence $\{v_{2k}\}$ in \eqref{eqn:a1}--\eqref{eqn:a2} converges to a single limit vector. The odd-index sequence $\{v_{2k+1}\}$ then also converges to a single limit vector, so that in the limit $k\rightarrow\infty$ the algorithm \eqref{eqn:a1}--\eqref{eqn:a2} cycles back and forth between two fixed limiting directions (that depend on $v_0$).

The main goal of the current paper is to prove that the same cyclic behavior of the restarted Arnoldi algorithm occurs for $s=2$, i.e., to prove the following theorem: 

\begin{theorem}\label{thm:main}
	If $A=A^T\in\R^{n\times n}$, $s=2$ and $d(A,v_0)\ge3$, then the even-index sequence $\{v_{2k}\}$ in \eqref{eqn:a1}--\eqref{eqn:a2} converges to a single limit vector. Hence also the odd-index sequence $\{v_{2k+1}\}$ converges to a single limit vector. 
\end{theorem}

Let $A=A^T\in\R^{n\times n}$, $s=2$, and $v\in\R^n$ with unit Euclidean norm and $d(A,v)\geq 3$. Then the Arnoldi projection \eqref{eqn:arn-proj} defines a uniquely determined monic polynomial $P_2(z;v)$ of degree~2, and we write
$$T_A(v):=\frac{P_2(A;v)v}{\|P_2(A;v)v\|}.$$
As already mentioned above, \cite[Theorem~1]{FLT} shows that then $d(A, T_A(v))\geq 3$, and hence there exists a uniquely determined monic polynomial of degree 4 given by
\begin{equation}\label{eqn:Q4}
	Q_4(z;v):=P_2(z;T_A(v))P_2(z;v).
\end{equation}
For simplicity, we denote quantities arising in \eqref{eqn:a1}--\eqref{eqn:a2} by
\begin{equation}\label{eqn:pk-Hk}
	p_k(z):=P_2(z;v_k)\quad\mbox{and}\quad H_k:=\|\widetilde{v}_{k+1}\|^2,
\end{equation}
then 
\begin{equation}\label{eqn:TA-qk}
	v_{k+1}=T_A(v_k)=\frac{p_k(A)v_k}{\sqrt{H_k}}
    \quad\mbox{and}\quad
	Q_4(z;v_k)=p_{k+1}(z)p_k(z)=:q_k(z).
\end{equation}
We will use the following further facts from~\cite{FLT}:

\begin{enumerate}
	\item[(F1)] $H_k\uparrow H>0$ for $k\rightarrow\infty$~\cite[eqn. (17) with $H=\tau^2$]{FLT}.
	\item[(F2)] If $H_{k+1}=H_k$ at some step $k\geq 0$, then $v_{k+2}=v_k$ and both parity sequences remain constant from this step onwards~\cite[start of Section~4.1]{FLT}. 
	\item[(F3)] The set of limit vectors $\Sigma_A^*$ of the even-index sequence $\{v_{2k}\}$ is nonempty, closed, and connected.  
    Each $v_*\in \Sigma_A^*$ satisfies $T_A^2(v_*)=v_*$, $(Q_4(A;v_*)-HI_n)v_*=0$, and $3\leq d(A,v_*)\leq 4$~\cite[Theorems~3 and~5]{FLT}. The same holds for the set of limit vectors of the odd-index sequence $\{v_{{2k+1}}\}$.
\end{enumerate}

We will now describe the strategy of the proof of Theorem~\ref{thm:main} given in Sections~\ref{sec:polys}--\ref{sec:S4}. The overall approach is motivated by the discussion in~\cite[pp.~13--15]{FLT}. In the first step, we show that for any given initial vector $v_0$ with $d(A,v_0)\geq 3$, the polynomial sequences $\{p_{2k}\}$ and $\{p_{2k+1}\}$ converge to uniquely determined limit polynomials $p_*$ and $q_*$; see Section~\ref{sec:polys}. Showing that these sequences converge is also a major ingredient in the approach of Zhuk and Bondarenko to the original Forsythe conjecture for $s=2$~\cite{ZB}. As described in~\cite[pp.~14--15]{FLT}, the remaining part of their approach is incomplete.

The essential observation in Section~\ref{sec:active} is that convergence of the polynomial sequences implies that the ``active eigenvalues'' of every even- or odd-index limit vector belong to the same fixed set $S$, which contains either three or four distinct eigenvalues of $A$. We then use normalization and the orthogonality conditions to obtain a linear system for the squared coefficients of each even-index limit vector $v_*$ with respect to the corresponding eigenvectors. For each fixed three-element subset of $S$, this system has a fixed nonsingular $3\times 3$ matrix and a fixed right-hand side (independent of $v_*$), so the squared coefficients are uniquely determined. Since there are only finitely many such subsets and only finitely many sign choices, $\Sigma_A^*$ contains only finitely many vectors of grade $3$. If $|S|=3$, connectedness of $\Sigma_A^*$ therefore implies that it is a singleton, which completes the proof in this case; see Section~\ref{sec:S3}. The case $|S|=4$ is more challenging, since normalization and the orthogonality conditions for each $v_*$ with $d(A,v_*)=4$ only yield an underdetermined system with $3$ equations for the $4$ squared coefficients. The main work then is to obtain an additional linear equation, which we achieve by showing that the sequence of the Rayleigh quotients $\{v_{2k}^TAv_{2k}\}$ converges. The resulting system then has a fixed nonsingular $4\times 4$ matrix and a fixed right-hand side (both independent of $v_*$), which again shows that $\Sigma_A^*$ is a singleton; see Section~\ref{sec:S4}.

\section{Convergence of the polynomial sequences}\label{sec:polys}

Our goal in this section is to show that the polynomial sequences $\{p_{2k}\}$ and $\{p_{2k+1}\}$ converge to uniquely determined limit polynomials. All polynomials that occur below are real. Constants denoted by $C>0$ may change from one occurrence to the next, but are independent of $k$.

For two polynomials $f,g$, define
\[
 \ip{f}{g}_k:=(g(A)v_k)^T f(A)v_k.
\]
This is a positive semidefinite form on all polynomials and an inner
product on the polynomials of degree smaller than $d(A,v_k)$; cf., e.g.,~\cite[Corollary~3.4.2]{LieStr13}. We write
$\|f\|_k:=\ip{f}{f}_k^{1/2}$. Using $v_{k+1}=p_k(A)v_k/\sqrt{H_k}$ we obtain
\begin{align}
 \ip{f}{g}_{k+1}&=\frac{\ip{p_kf}{p_kg}_k}{H_k},\label{eqn:ip1}\\
 \ip{f}{g}_{k+2}&=\frac{\ip{q_kf}{q_kg}_k}{H_kH_{k+1}}.
 \label{eqn:ip2}
\end{align}

For polynomials of degree at most three, $f(z)=\sum_{j=0}^3\alpha_jz^j$, we consider the norm $\|f\|:=\|[\alpha_0,\ldots,\alpha_3]^T\|$. Let $A=UDU^T$ be an eigendecomposition of $A$ with $D={\rm diag}(\lambda_1,\dots,\lambda_n)$ and an orthogonal matrix $U\in\R^{n\times n}$. Then $\|Uv_k\|=1$ and the Cauchy-Schwarz inequality yields 
\begin{align*}
\|f\|_k^2 &=(U^Tv_k)^Tf(D)^2 (U^Tv_k)\leq \max_{1\leq i\leq n} |f(\lambda_i)|^2=\max_{1\leq i\leq n} \left|\sum_{j=0}^3 \alpha_j \lambda_i^j\right|^2 \\
&\leq \max_{1\leq i\leq n} \left(\sum_{j=0}^3 \lambda_i^{2j}\right)\cdot \left(\sum_{j=0}^3 \alpha_j^2\right)
=C^2 \|f\|^2,\quad \mbox{where}\quad C^2:=\max_{1\leq i\leq n} \left(\sum_{j=0}^3 \lambda_i^{2j}\right).
\end{align*}
Thus, we have the uniform upper bound $\|f\|_k\le C\|f\|$.

\begin{lemma}\label{lem:polys}
The polynomials $p_k$ in \eqref{eqn:pk-Hk} have uniformly bounded coefficients, and there is a constant $C>0$ such that
\begin{equation}\label{eq:variation}
 \|p_{k+2}-p_k\|\le C(H_{k+1}-H_k),\quad k\ge0.
\end{equation}
Both sequences $\{p_{2k}\}$ and $\{p_{2k+1}\}$ converge.
\end{lemma}

\begin{proof}
Write $p_k(z)=z^2+b_kz+a_k$. The orthogonality conditions \eqref{eqn:arn-proj} for $p_k$ are
\begin{equation}\label{eq:orthogonality}
 \ip{p_k}{1}_k=\ip{p_k}{z}_k=0.
\end{equation}
They are equivalent to the linear system
\[
 G_k\begin{bmatrix}a_k\\b_k\end{bmatrix}
 =-\begin{bmatrix}\ip{z^2}{1}_k\\\ip{z^3}{1}_k\end{bmatrix},
 \quad
 G_k:=\begin{bmatrix}
 \ip{1}{1}_k&\ip{z}{1}_k\\
 \ip{z}{1}_k&\ip{z}{z}_k
 \end{bmatrix}.
\]
Here $G_k$ is the Gram matrix of the two vectors $v_k, Av_k$. Since every vector of the sequence \eqref{eqn:a1}--\eqref{eqn:a2} has grade at least three, $G_k$ is positive definite. Moreover, every limit vector also has grade at least three, and hence compactness gives a uniform bound $G_k\ge cI_2$ for some $c>0$. Also note that the right-hand side of the system above is uniformly bounded, so that the coefficients $a_k,b_k$ are uniformly bounded as well. For any polynomial $f(z)=bz+a$ we have 
$$\|f\|_k^2=\ip{f}{f}_k=b^2\ip{z}{z}_k+2ab\ip{z}{1}_k+a^2\ip{1}{1}_k=[a,b]G_k[a,b]^T.$$ 
Therefore,
\begin{equation}\label{eq:linear-lower-bound}
 c\|f\|^2\leq \|f\|_k^2,
\end{equation}
for every polynomial $f$ of degree at most one and every $k\geq 0$, which we will use below.

Since $p_{k+1}-p_k$ has degree at most one, orthogonality gives
\begin{equation}\label{eqn:qk-1}
 \ip{q_k}{1}_k=\ip{p_k}{p_{k+1}}_k
 =\ip{p_k}{p_k}_k=H_k.
\end{equation}
Equation~\eqref{eqn:ip2} with $f=g=1$ yields $\ip{q_k}{q_k}_k=H_kH_{k+1}$. Define $r_k:=q_k-H_k$. Then
\begin{equation}\label{eq:energy}
 \|r_k\|_k^2
 =\ip{q_k}{q_k}_k-2H_k\ip{q_k}{1}_k+H_k^2
 =H_k(H_{k+1}-H_k).
\end{equation}
For any polynomial $f$ of degree at most one, \eqref{eqn:ip1} and
orthogonality imply that
\[
 \ip{p_kf}{q_k}_k
 =\ip{p_kf}{p_kp_{k+1}}_k
 =H_k\ip{f}{p_{k+1}}_{k+1}=0,
\]
and hence
\begin{equation}\label{eq:double-orthogonality}
 \ip{p_kf}{r_k}_k
 =\ip{p_kf}{q_k}_k-H_k\ip{p_k}{f}_k=0.
\end{equation}
Define $D_k:=p_{k+2}-p_k$, which has degree at most one. We can assume that $D_k\neq 0$, because for $D_k=0$ the bound \eqref{eq:variation} is trivial. Orthogonality and \eqref{eqn:ip2} yield
\[
 \ip{D_k}{f}_{k+2}
 =-\ip{p_k}{f}_{k+2}
 =-\frac{\ip{p_kf}{q_k^2}_k}{H_kH_{k+1}}.
\]
Using $q_k^2=(r_k+H_k)^2=r_k^2+2r_kH_k+H_k^2$, as well as \eqref{eq:orthogonality} and
\eqref{eq:double-orthogonality} shows that
\begin{equation}\label{eq:factor-difference}
 \ip{D_k}{f}_{k+2}
 =-\frac{\ip{p_kf}{r_k^2}_k}{H_kH_{k+1}}.
\end{equation}
Considering the term in the numerator on the right-hand side with $f=D_k$, an evaluation on the fixed spectrum of $A$ gives
\[
 |\ip{p_kD_k}{r_k^2}_k|
 \le\max_{\lambda\in\operatorname{spec}(A)}
             |p_k(\lambda)D_k(\lambda)|\,\|r_k\|_k^2
 \le C\|D_k\|\|r_k\|_k^2,
\]
where the last inequality uses the uniform coefficient bound for $p_k$. Combining  \eqref{eq:linear-lower-bound}, \eqref{eq:energy}, and \eqref{eq:factor-difference}
we obtain
\[
 c\|D_k\|^2
 \le\|D_k\|_{k+2}^2
 =\frac{|\ip{p_kD_k}{r_k^2}_k|}{H_kH_{k+1}}
 \le C\|D_k\|(H_{k+1}-H_k).
\]
Here $H_{k+1}\ge H_0>0$ has been absorbed into the constant. Division by $\|D_k\|$ now proves \eqref{eq:variation}. Summing \eqref{eq:variation} gives
\[
\sum_{k=0}^{\infty}\|p_{k+2}-p_k\|
\leq C\sum_{k=0}^{\infty}(H_{k+1}-H_k)
= C(H-H_0)<\infty.
\]
Thus, for each $r\in\{0,1\}$, the series
\[
p_r+\sum_{j=0}^{\infty}
\bigl(p_{2j+r+2}-p_{2j+r}\bigr)
\]
converges absolutely in the coefficient norm. Its partial sums are
$p_{2m+r}$, so both the even- and the odd-index polynomial sequences
converge.
\end{proof}

Lemma~\ref{lem:polys} shows that
\begin{equation}\label{eq:polynomial-limits}
 p_{2k}\rightarrow p_*\quad\mbox{and}\quad p_{2k+1}\rightarrow q_*
\end{equation}
for uniquely determined monic quadratic polynomials $p_*,q_*$.
Since $q_k=p_{k+1}p_k$, we also have
\begin{equation}\label{eq:product-limit}
 q_k\rightarrow p_*q_*=:Q_*.
\end{equation}
Moreover,
\[
 q_{k+1}-q_k=p_{k+1}(p_{k+2}-p_k).
\]
The uniform coefficient bound and \eqref{eq:variation} imply
$\|q_{k+1}-q_k\|\le C(H_{k+1}-H_k)$. Summation from $k$ to $\infty$
yields the estimate
\begin{equation}\label{eq:tail}
 \|q_k-Q_*\|\le C(H-H_k).
\end{equation}
(Note that $q_k-Q_*$ has degree at most $3$.)

Let $v_*\in\Sigma_A^*$ and choose an even-index subsequence $v_{2k_j}\rightarrow v_*$. Passing to the limit and using (16) and uniqueness of the Arnoldi projection gives $P_2(z;v_*)=p_*(z)$, and hence \eqref{eqn:TA-qk} yields 
$$v_{2k_j+1}=\frac{p_{2k_j}(A)v_{2k_j}}{\sqrt{H_{2k_j}}}\rightarrow \frac{p_*(A)v_*}{\sqrt{H}}=T_A(v_*).$$
Applying the same argument to any subsequence of odd indices yields $P_2(z; T_A(v_*))=q_*(z)$, and hence 
$$Q_4(z; v_*) = p_*(z)q_*(z) = Q_*(z).$$
For a limit vector of the sequence $\{v_{2k+1}\}$ the two factors are interchanged, and their product is again $Q_*$. Thus the polynomial in fact (F3) is the same for every even- or odd-index limit vector. In particular, (F3) gives
\begin{equation}\label{eq:limit-equation}
 (Q_*(A)-HI_n)v_*=0,
\end{equation}
for every limit vector.

\section{Active and inactive eigenvalues}\label{sec:active}

If $H_{k+1}=H_k$ at any step $k$, then the convergence of the sequences follows from fact (F2). We may therefore assume from now on that \begin{equation}\label{eq:strict-height}
 H_k<H\qquad\text{for every }k\geq 0.
\end{equation}
Suppose that $d(A,v_0)=m\ge3$. Then
\[
 v_0=\sum_{j=1}^m\nu_{0,j}u_j,\quad
 \sum_{j=1}^m\nu_{0,j}^2=1,\quad \nu_{0,j}\ne0, \quad j=1,\dots,m,
\]
where $u_1,\dots,u_m$ are mutually orthonormal eigenvectors corresponding to distinct eigenvalues $\mu_1,\dots,\mu_m$ of $A$. We call these the eigenvalues \emph{active in $v_0$}. All other eigenvalues of $A$ are called \emph{inactive in $v_0$}. Every vector $v_k$ in \eqref{eqn:a1}--\eqref{eqn:a2} is given by some polynomial in $A$ times $v_0$, and hence is of the form
\begin{equation}\label{eq:spectral-coordinates}
 v_k=\sum_{j=1}^m\nu_{k,j}u_j,\quad
 \sum_{j=1}^m\nu_{k,j}^2=1.
\end{equation}
If $\nu_{k,j}\neq 0$, then the eigenvalue $\mu_j$ is called \emph{active in $v_k$}. All other eigenvalues of $A$ are called \emph{inactive in $v_k$}. 

In terms of the coefficients, the equation $v_{k+1}=p_k(A)v_k/\sqrt{H_k}$ for $k\geq 0$ translates into
\[
 \nu_{k+1,j}=\frac{p_k(\mu_j)}{\sqrt{H_k}}\nu_{k,j},\quad j=1,\dots,m.
\]
Thus, no eigenvalue that is inactive in $v_k$ can become active in $v_{k+1}$, and any eigenvalue that is active in $v_0$ can become inactive at most once during the iteration. Also note that every $v_k$ has at least three active eigenvalues, since $d(A,v_k)\ge3$.

When we study the limiting behavior of \eqref{eqn:a1}--\eqref{eqn:a2} we can discard a finite initial segment ending at some index $k$ after the last ``deactivation'' of a previously active eigenvalue. After relabeling the remaining
eigenvalues and eigenvectors, we keep the notation $m,\mu_j,u_j,\nu_{k,j}$ for this restricted iteration. We may then assume that every remaining coefficient $\nu_{k,j}$ is nonzero at every finite step under consideration.

Every (even- or odd-index) limit vector $v_*=\sum_{j=1}^m \nu_{*,j}u_j$ satisfies \eqref{eq:limit-equation}, and hence
$$(Q_*(\mu_j)-H)\nu_{*,j}=0,\quad j=1,\dots,m.$$
For every $\nu_{*,j}\neq 0$ we require $Q_*(\mu_j)-H=0$. The polynomial $Q_*(z)-H$ is monic of degree $4$, hence has at most $4$ distinct roots, and does not depend on $v_*$. Since $3\leq d(A,v_*)\leq 4$ holds for every $v_*\in\Sigma_A^*$, we have the following result that is essential for the rest of the proof.

\begin{lemma}
The active eigenvalues in every (even- or odd-index) limit vector are contained in the same fixed set
$$S:=\{\mu_j:Q_*(\mu_j)=H\},\quad\mbox{where}\quad 3\leq |S|\leq 4.$$
\end{lemma}

\section{Limit vectors with grade $3$}\label{sec:S3}

Possibly after relabeling the eigenvectors, any $v_*\in\Sigma_A^*$ with $d(A,v_*)=3$ can be written as  $v_*=w_1u_1+w_2u_2+w_3u_3$ , where the corresponding active eigenvalues are $\{\lambda_1,\lambda_2,\lambda_3\}\subseteq S$. Normalization and the orthogonality conditions (cf. \eqref{eq:orthogonality} for $v_*$) give
\[
 \sum_{i=1}^3w_i^2=1,\qquad
 \sum_{i=1}^3w_i^2p_*(\lambda_i)=0,\qquad
 \sum_{i=1}^3w_i^2\lambda_i p_*(\lambda_i)=0.
\]
Since $Q_*(\lambda_i)=p_*(\lambda_i)q_*(\lambda_i)=H>0$, these equations are
equivalent to
\[
 \begin{bmatrix}
 q_*(\lambda_1)/H&q_*(\lambda_2)/H&q_*(\lambda_3)/H\\
 1&1&1\\
 \lambda_1&\lambda_2&\lambda_3
 \end{bmatrix}
 \begin{bmatrix}
 p_*(\lambda_1)&0&0\\
 0&p_*(\lambda_2)&0\\
 0&0&p_*(\lambda_3)
 \end{bmatrix}
 \begin{bmatrix}w_1^2\\w_2^2\\w_3^2\end{bmatrix}
 =\begin{bmatrix}1\\0\\0\end{bmatrix}.
\]
Considering the first matrix, we note that the polynomials $1,z,q_*(z)/H$ form a basis of the vector space of the (real) polynomials of degree at most two, and evaluation at three distinct points yields a nonsingular matrix. The second (diagonal) matrix is obviously nonsingular. Thus, the three squared coefficients $w_i^2$ are uniquely determined. Only finitely many sign choices for the coefficients $w_i$ of $v_*$ remain, and there are only finitely many subsets of $S$ with three elements. Hence $\Sigma_A^*$ contains only finitely many vectors $v_*$ with $d(A,v_*)=3$. 

If $|S|=3$, then all vectors in $\Sigma_A^*$ have grade three, and thus $\Sigma_A^*$ is a finite set. Since $\Sigma_A^*$ is also closed and connected (see fact (F3)), it must be a singleton, i.e., the sequence $\{v_{2k}\}$ converges. The convergence of $\{v_{2k+1}\}$ immediately follows from \eqref{eq:polynomial-limits} and
$v_{2k+1}=p_{2k}(A)v_{2k}/\sqrt{H_{2k}}$.

\section{Limit vectors with grade $4$}\label{sec:S4}

It remains to consider $|S|=4$, i.e.,
$$S=\{\lambda_1<\lambda_2<\lambda_3<\lambda_4\}.$$
The definition of $S$ gives the four interpolation conditions $Q_*(\lambda_i)=H$, $i=1,\dots,4$. Since $Q_*(z)-H$ is monic of degree 4 and vanishes at the four points
of $S$, we have
\begin{equation}\label{eq:four-factor}
 Q_*(z)=p_*(z)q_*(z)=\Pi(z)+H,\quad\mbox{where}\quad \Pi(z):=\prod_{i=1}^4(z-\lambda_i).
\end{equation}
Note that this is the interpolation argument following~\cite[Theorem~5]{FLT}, applied to the common polynomial $Q_*$. If $|\Sigma_A^*|=1$, then the proof is complete. Otherwise, the connectedness of $\Sigma_A^*$ and the fact that this set contains only finitely many vectors of grade~$3$ show that $\Sigma_A^*$ contains a vector of grade~$4$. We assume this in the following.

If $v_*=\sum_{i=1}^4 w_iu_i\in\Sigma_A^*$ has $d(A,v_*)=4$, then similarly to the above, normalization and orthogonality conditions give the linear algebraic system
\[
 \begin{bmatrix}
 1&1&1&1\\
 p_*(\lambda_1)&p_*(\lambda_2)&p_*(\lambda_3)&p_*(\lambda_4)\\
 \lambda_1p_*(\lambda_1)&\lambda_2p_*(\lambda_2)&
 \lambda_3p_*(\lambda_3)&\lambda_4p_*(\lambda_4)
 \end{bmatrix}
 \begin{bmatrix}w_1^2\\w_2^2\\w_3^2\\w_4^2\end{bmatrix}
 =\begin{bmatrix}1\\0\\0\end{bmatrix}.
\]
Unlike in the case $d(A,v_*)=3$, however, the squared coefficients $w_i^2$ are not uniquely determined through this system. The main goal in the following is to obtain one further linear equation. This will be done by considering the sequence of the Rayleigh quotients
\begin{equation}\label{eqn:RQ}
\rho_k:=v_k^TA v_k=\ip{z}{1}_k,
\end{equation}
and showing that $\rho_{2k}\rightarrow \rho_*$ (independent of $v_*$). The additional equation is then given by
$$v_*^TAv_*=\sum_{i=1}^4 w_i^2 \lambda_i=\rho_*,$$
which yields a nonsingular system determining the $w_i^2$ uniquely. 

For the eigenvalues that are active in $v_0$ but \emph{not} contained in $S$, we retain the notation $\mu_j$ for clarity of the presentation. We call these $\mu_j$ the \emph{exterior eigenvalues} (with respect to $S$), and the coefficients of any $v_k$ corresponding to these eigenvalues are called the \emph{exterior coefficients of $v_k$}. Since the active eigenvalues of every (even- or odd-index) limit vector are contained in the set $S$, the exterior coefficients of the vectors $v_k$ must converge to zero. The following result shows how the decay of these exterior coefficients can be controlled.    

\begin{lemma}\label{lem:exterior-decay}
Assume that $|S|=4$ and that $\Sigma_A^*$ contains a vector of grade 4. Then there are constants $C>0$ and $0<\vartheta<1$ such that, for all sufficiently large $k$,
\begin{equation}\label{eq:exterior-decay}
 \varepsilon_k:=\sum_{\mu_j\notin S}\nu_{k,j}^2
 \le C\vartheta^k.
\end{equation}
\end{lemma}

\begin{proof}
If $v_k$ has no exterior eigenvalues, then $\varepsilon_k=0$ and the assertion is immediate. Otherwise, suppose that $v_k$ is as in \eqref{eq:spectral-coordinates} with $\nu_{k,j}\neq 0$ for some exterior eigenvalue $\mu_j\notin S$. By construction,
\begin{equation}\label{eq:component}
 \nu_{k+2,j}
 =\frac{q_k(\mu_j)}{\sqrt{H_kH_{k+1}}}\nu_{k,j}.
\end{equation}
We know from fact (F1) that $\sqrt{H_k H_{k+1}}\uparrow H$, and from \eqref{eq:product-limit} that $q_k(\mu_j)\rightarrow Q_*(\mu_j)$. If $|Q_*(\mu_j)|>H$, then the absolute value of the multiplier in \eqref{eq:component} will eventually (for sufficiently large $k$) become larger than $1$, contradicting that the exterior coefficients converge to zero. Thus, we must have $|Q_*(\mu_j)|\leq H$. The equality $Q_*(\mu_j)=H$ is excluded by the definition of $S$ (since $\mu_j$ is an exterior eigenvalue). We will now show that also $Q_*(\mu_j)\neq -H$, which implies $|Q_*(\mu_j)|<H$.  

Using the polynomial $\Pi$ as in \eqref{eq:four-factor} we consider the standard Lagrange basis polynomials
$$L_i(z):=\frac{\Pi(z)}{(z-\lambda_i)\Pi'(\lambda_i)},\quad i=1,\dots ,4,$$
which give $g(z)=\sum_{i=1}^4 g(\lambda_i)L_i(z)$ for every polynomial $g$ of degree at most $3$. Thus, 
\begin{equation}\label{eq:lagrange-exterior}
g(\mu_j)=\sum_{i=1}^4 g(\lambda_i) \beta_{i,j},\quad 
\beta_{i,j}=\frac{\Pi(\mu_j)}{(\mu_j-\lambda_i)\Pi'(\lambda_i)},\;\; i=1,\dots,4,\quad \sum_{i=1}^4\beta_{i,j}=1,
\end{equation}
where the last equality follows from taking $g(z)=1$.

As described above, for sufficiently large $k$ all coefficients of $v_k$ in its expansion \eqref{eq:spectral-coordinates} are assumed nonzero, since we can consider the iteration after the last ``deactivation'' of an active eigenvalue. For any exterior coefficient $\nu_{k,j}$ we may therefore define
$$Y_{j,k}:=\log\nu_{k,j}^2
       +\sum_{i=1}^4\beta_{i,j}\log\nu_{k,i}^2,$$
where $\nu_{k,i}$, $i=1,\dots,4$, are the coefficients corresponding to the active eigenvalues in $S$. Taking logarithms of the squared recurrence \eqref{eq:component} gives
$$\frac{\log\nu_{k+2,j}^2-\log\nu_{k,j}^2}{2}
=\log|q_k(\mu_j)|-\frac12(\log H_k+\log H_{k+1}).$$
The same relation holds for each coefficient corresponding to $\lambda_i$.
Substitution into the definition of $Y_{j,k}$ therefore yields
\begin{align*}
\frac{Y_{j,k+2}-Y_{j,k}}2
={}&\log|q_k(\mu_j)|
+\sum_{i=1}^4\beta_{i,j}\log|q_k(\lambda_i)|
-\frac{1+\sum_{i=1}^4\beta_{i,j}}2
 (\log H_k+\log H_{k+1}).
\end{align*}
Using $\sum_{i=1}^4\beta_{i,j}=1$ from \eqref{eq:lagrange-exterior}, and
$q_k(\lambda_i)>0$ for all sufficiently large $k$ (which holds since $q_k(\lambda_i)\to H>0$), we obtain
\begin{equation}\label{eq:log-difference}
\frac{Y_{j,k+2}-Y_{j,k}}2
=\log|q_k(\mu_j)|+\sum_{i=1}^4\beta_{i,j}\log q_k(\lambda_i)
-\log H_k-\log H_{k+1}.
\end{equation}
Write $h_k:=H-H_k$ and $e_k:=q_k-Q_*$, which has degree at most $3$. Then, by \eqref{eq:tail}, $\|e_k\|=O(h_k)$. Recall that we want to show $Q_*(\mu_j)\neq -H$. Assume the contrary, i.e., that $Q_*(\mu_j)=-H$. Then $q_k(\mu_j)=-H+e_k(\mu_j)$, and a Taylor expansion at $H>0$ gives
\begin{align*}
 \log|q_k(\mu_j)|
  &=\log H-\frac{e_k(\mu_j)}H+O(h_k^2),\\
 \log q_k(\lambda_i)
  &=\log H+\frac{e_k(\lambda_i)}H+O(h_k^2),\\
 -\log H_k-\log H_{k+1}
  &=-2\log H+\frac{h_k+h_{k+1}}H+O(h_k^2).
\end{align*}
Thus, in \eqref{eq:log-difference} the constant terms and the terms involving $e_k$ cancel, so that
\begin{equation}\label{eq:log-increase}
 \frac{Y_{j,k+2}-Y_{j,k}}2
 =\frac{h_k+h_{k+1}}H+O(h_k^2)>0
\end{equation}
for all sufficiently large $k$. Here $0<h_{k+1}\le h_k\to0$ by \eqref{eq:strict-height}, so the positive term is at least $h_k/H$ and dominates the remainder. Along any even subsequence converging to a limit vector with grade $4$, the weights $\nu_{k,i}^2$, $i=1,\dots,4$, tend to positive values, whereas $\nu_{k,j}^2\to0$. Thus $Y_{j,k}\to-\infty$ along this subsequence, which contradicts \eqref{eq:log-increase} and shows that $Q_*(\mu_j)\neq -H$.

Consequently, every remaining exterior eigenvalue $\mu_j$ satisfies $|Q_*(\mu_j)|<H$. Choose a constant $\vartheta$ such that
\[
\max_{\mu_j\notin S}\frac{|Q_*(\mu_j)|}{H}<\vartheta<1.
\]
Since the two-step multipliers in \eqref{eq:component} converge to $Q_*(\mu_j)/H$, and there are only finitely many exterior eigenvalues, there is an index $K$
such that
$$\frac{|q_k(\mu_j)|}{\sqrt{H_kH_{k+1}}}\le\vartheta,\quad\mbox{for every}\; 
k\ge K\;\mbox{and}\; \mu_j\notin S.$$
Squaring \eqref{eq:component} and summing over the exterior eigenvalues gives $\varepsilon_{k+2}\le\vartheta^2\varepsilon_k$ for $k\ge K$. Iterating separately on the two parities yields
\[
\varepsilon_{K+r+2\ell}\le\vartheta^{2\ell}\varepsilon_{K+r}
\quad\mbox{for}\; r\in\{0,1\}\;\mbox{and}\; \ell\ge0.
\]
Hence, choosing a positive constant $C$ no smaller than
\[
\max_{r\in\{0,1\}}\left\{\vartheta^{-(K+r)}\varepsilon_{K+r}\right\},
\]
we obtain $\varepsilon_k\le C\vartheta^k$ for every $k\ge K$, which proves \eqref{eq:exterior-decay}.
\end{proof}

Define $\alpha_k:=[z^3](q_k-Q_*)$, the coefficient of $z^3$ in $q_k-Q_*$. By (17), $\alpha_k\to 0$. Since $Q_*=\Pi+H$ and $r_k=q_k-H_k$, we also have
\[
\alpha_k=[z^3](r_k-\Pi).
\]
To prove convergence of the even-index Rayleigh quotients, we shall establish a recurrence of the form
\[
\rho_{k+2}-\rho_k=-\kappa_k\alpha_k+E_k,
\quad 0\le\kappa_k\le C,\quad\sum_{k=0}^{\infty}|E_k|<\infty.
\]
Thus, we need to control the behavior of $\alpha_k$ beyond its convergence to zero. Lemma~\ref{lem1} bounds its successive changes. Together with $\alpha_k\to0$, this estimate will imply that either the sequence $\{\alpha_k\}$ decays exponentially, or it has a fixed sign eventually. Lemma~\ref{lem2} establishes
the recurrence above. In the exponential-decay case, the two-step Rayleigh quotient increments are absolutely summable. In the fixed-sign case, adding the tail of the error series yields a bounded, eventually monotone sequence. Both alternatives therefore imply convergence of $\{\rho_{2k}\}$.

\begin{lemma}\label{lem1}
For all sufficiently large $k$,
\begin{equation}\label{eq:lanczos-alpha}
|\alpha_{k+1}-\alpha_k|\le C(\alpha_k^2+\vartheta^k),
\end{equation}
where $\vartheta$ is as in Lemma~~\ref{lem:exterior-decay}.
\end{lemma}

\begin{proof}
For each $k$, let $f_k(z)=z^3+c_kz^2+b_kz+a_k$ be defined by
\begin{equation}
 \langle f_k,1\rangle_k
 =\langle f_k,p_k\rangle_k
 =\langle f_k,zp_k\rangle_k=0.
 \label{eq:lanczos-cubic}
\end{equation}
Expanding these conditions and using the orthogonality relations \eqref{eq:orthogonality}, together with the identities $\ip{1}{1}=1$, $\ip{z}{1}=\rho_k$, and $\ip{z^2}{p_k}=H_k$, we obtain the linear algebraic system
\[
B_k\begin{bmatrix}a_k\\b_k\\c_k\end{bmatrix}
=-\begin{bmatrix}\ip{z^3}{1}\\\ip{z^3}{p_k}\\\ip{z^3}{zp_k}\end{bmatrix},
\quad\mbox{where}\quad
B_k=\begin{bmatrix}
1&\rho_k&\ip{z^2}{1}\\
0&0&H_k\\
0&H_k&\ip{z^3}{p_k}
\end{bmatrix}.
\]
The identity $\ip{z^2}{p_k}=H_k$ follows because $z^2-p_k$ has degree at most one and is orthogonal to $p_k$. Since $\det B_k=-H_k^2$ and $H_k\ge H_0>0$, the system has a unique solution. Its entries and right-hand side are uniformly bounded, so $B_k^{-1}$ and the coefficients of $f_k$ are uniformly bounded as well.

The polynomial
\[
g_k:=r_k-\Pi-\alpha_k f_k
\]
has degree at most two. Equations \eqref{eqn:qk-1} and \eqref{eq:double-orthogonality} give
\[
\ip{r_k}{1}=\ip{r_k}{p_k}=\ip{r_k}{zp_k}=0.
\]
Since $\Pi$ vanishes on $S$ and all polynomial values involved are uniformly bounded on the fixed spectrum, we have
\[
|\ip{\Pi}{1}|+|\ip{\Pi}{p_k}|+|\ip{\Pi}{zp_k}|
\le C\varepsilon_k,
\quad\|\Pi\|_k^2\le C\varepsilon_k.
\]
The defining orthogonality conditions for $f_k$ and $r_k$ imply
\[
\ip{g_k}{1}=-\ip{\Pi}{1},\quad
\ip{g_k}{p_k}=-\ip{\Pi}{p_k},\quad
\ip{g_k}{zp_k}=-\ip{\Pi}{zp_k}.
\]
Writing $g_k(z)=g_{k,0}+g_{k,1}z+g_{k,2}z^2$, these relations form a system with the same coefficient matrix $B_k$. Its right-hand side is $O(\varepsilon_k)$, so $\|g_k\|\le C\varepsilon_k$ in the coefficient norm. Using $r_k=\Pi+\alpha_k f_k+g_k$, the uniform bounds above, and $\varepsilon_k\le1$, equation \eqref{eq:energy} now gives
\begin{equation}
 H_{k+1}-H_k=\frac{\|r_k\|_k^2}{H_k}
 \le C(\alpha_k^2+\varepsilon_k).
 \label{eq:lanczos-energy-bound}
\end{equation}
Write $p_k(z)=z^2+\beta_kz+\gamma_k$. Since
$q_{k+1}-q_k=p_{k+1}(p_{k+2}-p_k)$, comparison of the cubic coefficients and Lemma~\ref{lem:polys} yield
\[
|\alpha_{k+1}-\alpha_k|=|\beta_{k+2}-\beta_k|
\le C(H_{k+1}-H_k).
\]
Combining the last two inequalities with $\varepsilon_k\le C\vartheta^k$ proves the assertion.
\end{proof}

We now define
\begin{equation}\label{eqn:gram-det}
\Delta_k:=\det G_k,\quad G_k:=[\ip{z^i}{z^j}]_{i,j=0}^3.
\end{equation}
This is the Gram matrix of $v_k,Av_k,A^2v_k,A^3v_k$, so $\Delta_k\ge0$, with equality if and only if $d(A,v_k)\le3$. Its entries, and hence $\Delta_k$, are uniformly bounded.

\begin{lemma}\label{lem2}
For all sufficiently large $k$,
\begin{align}
 \rho_{k+2}-\rho_k
 &=-\kappa_k\alpha_k+E_k\quad\mbox{with}\quad \kappa_k:=\frac{2\Delta_k}{H_k^2H_{k+1}},\quad 0\leq \kappa_k\leq C,\label{eq:lanczos-rho} \\
 |E_k| &\leq C(H_{k+1}-H_k+\vartheta^k)\quad\mbox{and}\quad \sum_{k=0}^\infty |E_k|<\infty,
\label{eq:lanczos-rho1}
\end{align}
and where \(\vartheta\) is as in Lemma~\ref{lem:exterior-decay}.
\end{lemma}

\begin{proof}
Let $f_k$ and $g_k$ be the auxiliary polynomials from the proof of Lemma~\ref{lem1}. Write $p_k(z)=z^2+\beta_kz+\gamma_k$ and use the triangular matrices
\[
L_k=\begin{bmatrix}
1&0&0&0\\0&1&0&0\\\gamma_k&\beta_k&1&0\\0&\gamma_k&\beta_k&1
\end{bmatrix},\quad
R_k=\begin{bmatrix}
1&0&0&a_k\\0&1&0&b_k\\0&0&1&c_k\\0&0&0&1
\end{bmatrix}.
\]
Left multiplication by $L_k$ replaces the row polynomials
$(1,z,z^2,z^3)$ by $(1,z,p_k,zp_k)$, and right multiplication by $R_k$ replaces the column polynomials by $(1,z,z^2,f_k)$. Both matrices have determinant one. The last column of $L_kG_kR_k$ is
$[0,\ip{z}{f_k},0,0]^T$. Expansion along that column gives
\begin{equation}
\Delta_k =\langle z,f_k\rangle_k \det(B_k)=-H_k^2\langle z,f_k\rangle_k.
 \label{eq:lanczos-determinant}
\end{equation}
This identity also holds when the grade is three.

Since $\ip{z}{\Pi}=O(\varepsilon_k)$ and $\|g_k\|\le C\varepsilon_k$, the decomposition $r_k=\Pi+\alpha_k f_k+g_k$ gives
\begin{equation}
 \langle z,r_k\rangle_k
 =-\frac{\Delta_k}{H_k^2}\alpha_k+O(\vartheta^k).
 \label{eq:lanczos-first-moment}
\end{equation}
Equations \eqref{eqn:ip2} and \eqref{eq:energy}, together with $q_k=H_k+r_k$, yield the exact identity
\[
\rho_{k+2}-\rho_k
=\frac{2\ip{z}{r_k}}{H_{k+1}}
+\frac{\ip{z-\rho_k}{r_k^2}}{H_kH_{k+1}}.
\]
The first term equals $-\kappa_k\alpha_k+O(\vartheta^k)$.
Using $|\rho_k|\le\|A\|$ and \eqref{eq:energy}, the second term is bounded in absolute value by
\[
\frac{2\|A\|\|r_k\|_k^2}{H_kH_{k+1}}
=\frac{2\|A\|}{H_{k+1}}(H_{k+1}-H_k)
\le C(H_{k+1}-H_k).
\]
This proves the recurrence and the stated bound for $E_k$. Since
$H_k\uparrow H$ and $0<\vartheta<1$, telescoping the height increments and summing the geometric series give $\sum_k|E_k|<\infty$.
\end{proof}

Using Lemmas~\ref{lem1} and~\ref{lem2} we can now show that the sequence $\{\rho_{2k}\}$ converges.

\begin{lemma}
For any fixed $\eta\in(\vartheta,1)$, where \(\vartheta\) is as in Lemma~\ref{lem:exterior-decay}, exactly one of the following
alternatives holds:
\begin{enumerate}
\item[(i)] $|\alpha_k|<\eta^k$ for all sufficiently large $k$, and in this case $\alpha_k=O(\eta^k)$.
\item[(ii)] $|\alpha_k|\ge\eta^k$ for all sufficiently large $k$, and $\alpha_k$ has a fixed nonzero sign on this tail.
\end{enumerate}
Moreover, the sequence $\{\rho_{2k}\}$ has finite total variation, i.e., $\sum_{k=0}^{\infty}|\rho_{2k+2}-\rho_{2k}|<\infty$, and hence it converges.
\end{lemma}

\begin{proof}
We have $\alpha_k\to0$, and the estimate of Lemma~\ref{lem1} holds for all sufficiently large $k$. Choose $K_0$ so large that this estimate holds and that
\[
C|\alpha_k|+C(\vartheta/\eta)^k<1-\eta
\]
for all $k\geq K_0$. Such an index exists because both terms on the left tend to zero.

If $|\alpha_k|<\eta^k$ for every $k\ge K_0$, then alternative (i) holds. Otherwise, there is an index $k_0\ge K_0$ satisfying $|\alpha_{k_0}|\ge\eta^{k_0}$. At any index $k\ge K_0$ with $|\alpha_k|\ge\eta^k$, Lemma~\ref{lem1} gives
\[
\frac{|\alpha_{k+1}-\alpha_k|}{|\alpha_k|}
\le C|\alpha_k|+\frac{C\vartheta^k}{|\alpha_k|}
\le C|\alpha_k|+C(\vartheta/\eta)^k<1-\eta.
\]
Thus, $\alpha_{k+1}$ has the same sign as $\alpha_k$, and $|\alpha_{k+1}|>\eta|\alpha_k|\ge\eta^{k+1}$. Induction shows that the threshold remains satisfied and the sign remains unchanged for every $k\ge k_0$. Hence alternative (ii) holds. The two alternatives are mutually exclusive.

In order to show that the sequence $\{\rho_{2k}\}$ has finite total variation we consider any fixed $\eta\in(\vartheta,1)$. Note that finitely many initial terms do not affect the conclusion.

In alternative (i) we have $\alpha_k=O(\eta^k)$. The recurrence of Lemma~\ref{lem2},
the boundedness of $\kappa_k$, and the summability of $E_k$ imply, for
sufficiently large $J$,
\[
\sum_{k=J}^{\infty}|\rho_{2k+2}-\rho_{2k}|
\le C\sum_{k=J}^{\infty}|\alpha_{2k}|
+\sum_{k=J}^{\infty}|E_{2k}|<\infty.
\]
In alternative (ii), define
\[
\widehat\rho_k:=\rho_{2k}+\sum_{\ell=k}^{\infty}E_{2\ell}.
\]
The series is well defined by Lemma~\ref{lem2}, and $\widehat\rho_{k+1}-\widehat\rho_k=-\kappa_{2k}\alpha_{2k}$. Since $\kappa_{2k}\ge0$ and $\alpha_{2k}$ has a fixed sign eventually, $\widehat\rho_k$ is eventually monotone. The series is bounded because $|\rho_{2k}|\le\|A\|$ and the error tail tends to zero. Hence it converges and has finite total variation. Moreover,
\[
\rho_{2k+2}-\rho_{2k}
=(\widehat\rho_{k+1}-\widehat\rho_k)+E_{2k}.
\]
The summability of $E_{2k}$ therefore gives finite total variation of $\rho_{2k}$ as well, which finishes the proof.
\end{proof}

Finally, we write \(\rho_*:=\lim_{k\rightarrow\infty}\rho_{2k}\). Any
\(v_*\in\Sigma_*^A\) has the form
\(v_*=\sum_{i=1}^4w_i u_i\), with zero coefficients allowed. As outlined at the beginning of this section, normalization, the Rayleigh quotient, and the limiting
orthogonality conditions for $v_*$ give
\[
 \begin{bmatrix}
 1&1&1&1\\
 \lambda_1&\lambda_2&\lambda_3&\lambda_4\\
 p_*(\lambda_1)&p_*(\lambda_2)&p_*(\lambda_3)&p_*(\lambda_4)\\
 \lambda_1p_*(\lambda_1)&\lambda_2p_*(\lambda_2)&
 \lambda_3p_*(\lambda_3)&\lambda_4p_*(\lambda_4)
 \end{bmatrix}
 \begin{bmatrix}w_1^2\\w_2^2\\w_3^2\\w_4^2\end{bmatrix}
 =\begin{bmatrix}1\\\rho_*\\0\\0\end{bmatrix}.
\]
The polynomials \(1,z,p_*,zp_*\) form a basis of the vector space of polynomials
of degree at most three. Evaluation at four distinct points therefore yields a nonsingular matrix, so the squared coefficients \(w_i^2\) are uniquely determined. There are only finitely many sign choices, so $\Sigma_A^*$ is finite, and hence a singleton by fact (F3). Thus \(\{v_{2k}\}\) converges to a single limit vector \(v_*\), and 
\[
 v_{2k+1}=\frac{p_{2k}(A)v_{2k}}{\sqrt{H_{2k}}}
 \rightarrow\frac{p_*(A)v_*}{\sqrt H}
\]
gives also the convergence of the odd-index sequence, which finishes the proof of Theorem~\ref{thm:main}.

\section{Original proof, related developments, and the current paper}

\textbf{Original proof.} An initial proof of Theorem~\ref{thm:main} was found by the second author using Codex with GPT-5.6 Sol in Ultra reasoning. During their proof search, the agents were instructed to create independent notes on intermediate lemmas. Each document communicating a lemma was independently reviewed by two subagents for errors and notation quality. During this process, the following two strategies were particularly effective: explicitly claim that there is an error in the proof, and ask if the generated document was written by the same authors as~\cite{FLT}. The first instruction encouraged adversarial checking rather than confirmation of the proposed argument, while the second proved useful for detecting deviations from the notation and mathematical style of~\cite{FLT}. The main agent was instructed to iterate on this feedback until the reviewing agents found no major issues, which happened on July 26, 2026. The Codex-generated paper was posted after an initial (human) review on the second author's personal webpage on July 30, 2026\footnote{https://jarek.ai/papers/proof-of-the-forsythe-conjecture-for-s-2.pdf}. 

Working with Codex on the proof revealed two major failure modes:

(1) Agents easily got lost in computational experiments. Since \(\|v_{k+2} - v_k\|\rightarrow 0\) quickly, numerical experiments cannot indicate convergence. Alternatively, the agents ran dozens of symbolic experiments, attempting to find particular families for which convergence could be established. Experiments were therefore explicitly forbidden during the search for the proof, and the agents were instructed to solely rely on symbolic manipulation.

(2) Agents reliably drifted away from the original notation, often introducing unnecessary complexity. This complexity, paired with significant LLM slop, often obscured insights that would have been obvious if written with notation and in language used by human experts. Therefore, the notation was first fixed by running an audit of existing expert papers. Subsequently, this audit was used to anchor agents during their search.

\textbf{Related developments.} During our initial (human) review of the Codex-generated paper, we became aware of two proofs of the original Forsythe conjecture for $s=2$ generated with GPT-5.6 Sol almost at the same time. One was sent to us directly by Richard Peng (Carnegie Mellon University) on July 30, 2026, and one was published in arXiv on August 3, 2026~\cite{CST-s2}. These two papers consider a slightly different problem, and the proofs they give seemed to be considerably longer and more technical than the proof in our original Codex-generated paper. Subsequently, the authors of~\cite{CST-s2}  published a 150-page paper claiming a complete resolution of the original Forsythe conjecture on arXiv on September 4, 2026~\cite{CST-all}. According to that paper, the Forsythe conjecture is true for $s=1,2,3$ and false for $s\geq 4$. If this holds, we expect that for $A=A^T\in\R^{n\times n}$ the ACI($s$) conjecture posed in~\cite{FLT} is true for $s=3$ and false for $s\geq 4$. (Its truth for $s=1$ is shown in~\cite{FLT}, and for $s=2$ in the current paper.)

\textbf{The current paper.} The current paper is an extensively rewritten and improved version of the one that was generated by the second author using Codex on July 26, 2026. Our motivation for writing it was to produce a proof of Theorem~\ref{thm:main} that is instructive and transparent, so that it can enhance the knowledge of interested human readers. Several paragraphs throughout the paper therefore explain the strategies and main ideas behind all steps of the proof. As described above, the agents were provided with an audit of the notation and terminology used in the existing expert literature. However, the Codex-generated paper still contained unnecessary or uncommon notation as well as several arguments that were too condensed or too technical to be readily comprehensible. Moreover, while studying the Codex-generated proof, we found a different strategy for completing the argument via the convergence of the Rayleigh quotients, which we present in Section~\ref{sec:S4}. In comparison with the Codex-generated proof this strategy is a better fit to the existing literature, because it is motivated by the interpolation argument following~\cite[Theorem~5]{FLT}. Moreover, we think that it is more elegant, and more natural in the given mathematical context. Thus, the proof presented here incorporates substantial mathematical improvements of the Codex-proof that were found during our subsequent analysis.

\end{document}